\documentclass[11pt,a4paper,epsf]{article}
\usepackage{latexsym}
\usepackage[utf8]{inputenc}
\usepackage{amsmath}
\usepackage{amsthm}
\usepackage{amsfonts}
\usepackage{amssymb}
\usepackage{epsfig}
\usepackage{nicefrac}
\usepackage[all]{xy}
\usepackage{graphicx}
\usepackage{enumerate}
\newtheorem{theorem}{Theorem}[section]
\newtheorem{corollary}[theorem]{Corollary}
\newtheorem{proposition}[theorem]{Proposition}
\newtheorem{lemma}[theorem]{Lemma}

\theoremstyle{definition}

\def\PP{\mathcal P}

\newcommand{\CC}{{\mathbb C}}

\newcommand{\RR}{{\mathbb R}}
\newcommand{\ZZ}{{\mathbb Z}}

\begin{document}

\title{Symplectic isotopies in dimension greater than four}

\author{R. Hind}

\date{\today}

\maketitle

\begin{abstract}
In any dimension $2n \ge 6$ we show that certain spaces of symplectic embeddings of a polydisk into a product $B^4 \times \RR^{2(n-2)}$ of a $4$-ball and Euclidean space, are not path connected. We also show that any pair of such nonisotopic embeddings can never be extended to the same ellipsoid.
\end{abstract}

\begin{section}{Introduction}

We study symplectic embeddings into Euclidean space $\RR^{2n}$, with coordinates $x_j,y_j$, $1 \le j \le n$, equipped with its standard symplectic form $\omega = \sum_{j=1}^n dx_j \wedge dy_j$. Often it is convenient to identify $\RR^{2n}$ with $\CC^n$ by setting $z_j = x_j + i y_j$. The basic domains for symplectic embedding problems are ellipsoids $E$ and polydisks $P$ which we define as follows.

$$E(a_1, \dots ,a_n) = \{\sum_j \frac{\pi |z_j|^2}{a_j} \le 1\};$$
$$P(a_1, \dots ,a_n) = \{\pi |z_j|^2 \le a_j \, \mathrm{for \, all} \, j\}.$$

These are subsets of $\CC^n$ and so inherit the symplectic structure. A ball of capacity $R$ is simply an ellipsoid $B^{2n}(R) = E(R, \dots ,R)$.

In this paper we will study isotopy classes of symplectic embeddings. We focus on the case when the dimension $2n \ge 6$ because this is much less understood than when our domains are $4$-dimensional.

We start by recalling that in dimension $4$ there are results showing that spaces of embeddings of polydisks are not path connected. The first was due to Floer, Hofer and Wysocki, showing that spaces of embeddings of a polydisk into a polydisk may be disconnected.

\begin{theorem} \label{fhw}(Floer-Hofer-Wysocki \cite{flho} Theorem $4$)
Let $\max(a,b) < R < a+b$. Then $g_0:(z_1,z_2) \mapsto (z_1,z_2)$ and $g_1:(z_1,z_2) \mapsto (z_2,z_1)$ give nonisotopic embeddings $P(a,b) \to P(R,R)$.
\end{theorem}

Note that if $R>a+b$ then the embeddings are isotopic in $P(R,R)$ through unitary maps. The condition $\max(a,b) < R$ ensures that the $g_i$ have images in $P(R,R)$.

The starting point for us is the following theorem about polydisks embedded in a ball.

\begin{theorem} \label{4dpoly} (Hind, \cite{hind} Theorem $1.1$)
There does not exist a Hamiltonian diffeomorphism $\phi$ with support contained in $B^4(2a+b)$ such that $\phi(P(a,b)) \subset \mathring{B}^4(a+b)$.
\end{theorem}

This leads immediately to examples of nonisotopic polydisks symplectomorphic to $P(a,b)$ with $b>2a$. Indeed, by a symplectic fold, for any $\epsilon>0$ there exists a symplectic embedding $P(a,b) \to B^4(2a+\frac{b}{2}+ \epsilon)$, see \cite{schl}, Proposition $4.3.9$.

It turns out that Theorem \ref{4dpoly} does have a generalization to higher dimensions, not only for polydisks but also for polylike domains (products of a disk and an ellipsoid) $Q$ which we define as follows.

$$Q(b,a_2,a_3, \dots ,a_n) = \{\pi|z_1|^2 \le b, \sum_{j=2}^n \frac{\pi |z_j|^2}{a_j} \le 1\}.$$

Then a generalization of Theorem \ref{4dpoly} is as follows. Note that by inclusion the polylike domain $Q(b,a_2, \dots ,a_n)$ sits inside $B^4(a_2+b) \times \RR^{2(n-2)}$.

\begin{theorem} \label{polylike}
Suppose that $a_2<b$ and $a_j > 2a_2$ for all $j \ge 3$.
There does not exist a Hamiltonian diffeomorphism $\phi$ with support contained in $\mathring{B}^4(2a_2+b) \times \RR^{2(n-2)}$ such that $\phi(Q(b,a_2, \dots ,a_n)) \subset \mathring{B}^4(a_2+b) \times \RR^{2(n-2)}$.
\end{theorem}

Theorem \ref{polylike} can be easily applied to give examples of nonisotopic polylike domains inside products $B^4(R) \times \RR^{2(n-2)}$.
The folding mentioned above applied to the first two complex coordinates gives a symplectic embedding $Q(b,a_2, \dots ,a_n) \to B^4(2a_2+\frac{b}{2}+ \epsilon)\times \RR^{2(n-2)}$ for any $\epsilon>0$, and if $2a_2<b$ then $\epsilon$ can be chosen such that $2a_2+\frac{b}{2}+ \epsilon < a_2+b$. Hence Theorem \ref{polylike} implies that this embedding cannot be symplectically isotopic to the inclusion. Similarly, if $a_3<a_2+b$ then by switching the $z_1$ and $z_3$ coordinates we get another embedding $Q(b,a_2, \dots ,a_n) \to \mathring{B}^4(a_2+b)\times \RR^{2(n-2)}$. Therefore we have the following corollary about embeddings of polylike domains.

\begin{corollary} Let $a_3, \dots ,a_n > 2a_2$ and choose $R$ with $a_2+b < R < 2a_2+b$. Suppose either $2a_2<b$ or $a_3 <a_2 + b$. Then there exists a symplectic embedding $Q(b,a_2, \dots ,a_n) \to B^4(R)\times \RR^{2(n-2)}$ which is not symplectically isotopic to the inclusion inside $B^4(R)\times \RR^{2(n-2)}$.
\end{corollary}

We note that our bound on $R$ is sharp in the sense that if $R>2a_2+b$ then the folding operation in the $(z_1,z_2)$ plane can be carried out in $B^4(R)$, see \cite{hind}, section $3$. Similarly if $R>a_3 + b$ then a rotation between the $z_1$ and $z_3$ coordinates can be carried out in $B^4(R)\times \RR^{2(n-2)}$. We can also produce examples of nonisotopic polydisks from Theorem \ref{polylike} by observing that $Q(b,a_2, \dots ,a_n) \subset P(b,a_2, \dots ,a_n)$, and so embeddings $P(b,a_2, \dots ,a_n) \to \mathring{B}^4(a_2+b) \times \RR^{2(n-2)}$ are also not isotopic to the inclusion. However it is also possible to work directly with higher dimensional polydisks and obtain a similar result.

\begin{theorem} \label{polydisk}  Let $a_1 \le \dots \le a_n$ with $a_3 > \max(2a_1,a_2)$ and $a_1+a_3 < R < 2a_1 + a_3$. Then the the space of embeddings $P(a_1, \ldots ,a_n) \to B^4(R) \times \Bbb R^{2(n-2)}$ is not path connected.

More precisely, the embedding $f:(z_1, z_2, z_3, \dots ,z_n) \mapsto (z_1, z_3, z_2, z_4, \dots ,z_n)$ is not isotopic to any map with image contained in $\mathring{B}^4(a_1+a_3) \times \Bbb R^{2(n-2)}$. In particular, the inclusion is not isotopic to $f$.
\end{theorem}

We remark that Theorem \ref{polydisk} generalizes Theorem \ref{fhw} in the case when $b>2a$ as follows.

\begin{corollary} Let $2a<b < R < a+b$. Then the two embeddings $g_0:(z_1,z_2) \mapsto (z_1,z_2)$ and $g_1:(z_1,z_2) \mapsto (z_2,z_1)$ from $P(a,b)$ into $B^2(R) \times \RR^2$ are not symplectically isotopic.
\end{corollary}

\begin{proof} Suppose to the contrary that such an isotopy exists. That is, for $0 \le t \le 1$ there exist symplectic embeddings $g_t(z_1,z_2) =(g^1_t(z_1,z_2),g^2_t(z_1,z_2))$ defined by maps $g^1_t, g^2_t: \CC^2 \to \CC$ with $g^1_0(z_1,z_2)=z_1$, $g^2_0(z_1,z_2)=z_2$, $g^1_1(z_1,z_2)=z_2$ and $g^2_1(z_1,z_2)=z_1$. Moreover we have $\pi|g^1_t(z_1,z_2)|^2 < a+b$ for all $t$ and all $(z_1,z_2) \in P(a,b)$.

Then we consider the isotopy of the polydisk $P=P(a,a,b)$ defined by $$f_t(z_1,z_2,z_3)=(z_1,g^1_t(z_2,z_3), g^2_t(z_2,z_3)).$$ Our bound on $|g^1_t|$ implies that $f_t(P) \subset \mathring{B}^4(2a+b)$ for all $t$. However $f_0$ is the inclusion and $f_1(z_1,z_2,z_3)=(z_1,z_3,z_2)$, contradicting Theorem \ref{polydisk}.
\end{proof}

In dimension $4$, that is, when $n=2$, a theorem of McDuff says that the space of symplectic embeddings of one ellipsoid into another is always path connected.

\begin{theorem} (McDuff \cite{mcduff} Corollary $1.6$, see also \cite{mcdops}) \label{ellisotopy}
For any $a,b,a',b'$, the space of symplectic embeddings $E(a,b) \to \mathring{E}(a',b')$ is path connected whenever it is nonempty.
\end{theorem}

This has the following corollary.

\begin{corollary}\label{mcdcor} If $f_0,f_1:P(a,b) \to \mathring{B}^4(R)$ are nonisotopic polydisks, then there does not exist an ellipsoid $E=E(a',b')$ such that $P(a,b) \subset E$ and both of the maps $f_0$ and $f_1$ extend to embeddings $E \to \mathring{B}^4(R)$.
\end{corollary}

We do not know if McDuff's Theorem \ref{ellisotopy} remains true in dimension greater than four. However Corollary \ref{mcdcor} does generalize, at least to the nonisotopic polydisk and polylike domains we consider in this paper.

\begin{theorem} \label{extn} \begin{enumerate}[i.]

\item Let $Q= Q(b,a_2, \dots ,a_n)$ be a polylike domain with $a_2<b$ and $a_j > 2a_2$ for all $j \ge 3$, and let $E=E(c_1, \dots ,c_n)$ be an ellipsoid with $Q \subset E$. Let $R<2a_2+b$. There do not exist embeddings $f_0,f_1:E \to B^4(R)\times \RR^{2(n-2)}$ such that $f_0$ restricts to the identity on $Q$ and $f_1(Q) \subset  \mathring{B}^4(a_2+b) \times \RR^{2(n-2)}$.

\item Let $P=P(a_1, \dots ,a_n)$ be a polydisk with $a_1 \dots \le a_n$ and $a_3 > \max(2a_1,a_2)$, and let $R$ satisfy $a_1+a_3 < R < 2a_1 + a_3$. Let $E=E(c_1, \dots ,c_n)$ be an ellipsoid with $P \subset E$. There do not exist embeddings $f_0,f_1:E \to B^4(R)\times \RR^{2(n-2)}$ such that $f_0|_P$ is given by $(z_1, z_2, z_3, \dots ,z_n) \mapsto (z_1, z_3, z_2, z_4, \dots ,z_n)$ and $f_1(P) \subset  \mathring{B}^4(a_1+a_3) \times \RR^{2(n-2)}$.
\end{enumerate}
\end{theorem}

Applying Theorem \ref{extn} to specific examples of ellipsoids $E$, say such that a given map $f_0$ does extend to $E$, we can obtain various non-extension results for symplectic embeddings $f_1$.

Here is an example in dimension $4$.

\begin{proposition}\label{xtn}\begin{enumerate}[i.]

\item $E(2,4) \to \mathring{B}^4(R)$ if and only if $R \ge 4$;

\item $E(2,4) \cap \{\pi|z_2|^2 \ge 2 \} \to \mathring{B}^4(3.5)$;

\item the inclusion map $E(2,4) \cap \{\pi|z_2|^2 = 2 \} \to \mathring{B}^4(3.5)$ does not extend to an embedding of $E(2,4) \cap \{\pi|z_2|^2 \ge 2 \}$.
\end{enumerate}
\end{proposition}

The first two statements here follow for example from McDuff and Schlenk's classification \cite{mcdsch} of $4$-dimensional ellipsoid embeddings into balls, see Lemma \ref{lastlem}. The final statement is a consequence of our study of higher dimensional isotopy restrictions.

\vspace{0.2in}

{\bf Outline of the paper.}

The proof of Theorem \ref{polylike} is contained in section \ref{polyproof}. The techniques borrow heavily from the proof of Theorem \ref{4dpoly}, but with additional technicalities due to working in higher dimension. For these we rely on analysis from \cite{hindker}. The rough outline is as follows.

In section \ref{2pt1} we describe the basic arrangement. The product $B^4(R) \times \RR^{2(n-2)}$ for some $R<2a_2+b$ is partially compactified to $\CC P^2 \times \RR^{2(n-2)}$ and the polylike domain $Q$ is approximated by a smooth domain $W$. We argue by contradiction and assume that there exists a symplectic isotopy $W_t$, $0 \le t \le 1$, with $W_0=W$ and $W_1 \subset \mathring{B}^4(a_2+b) \times \RR^{2(n-2)}$.

Next, in section \ref{2pt2} the symplectic manifolds $X_t = (\CC P^2 \times \RR^{2(n-2)}) \setminus W_t$ are given almost-complex structures with cylindrical ends and we compute index and area formulas for finite energy holomorphic curves. We refer to the series of papers of Hofer, Wysocki and Zehnder, \cite{three}, \cite{hofa}, \cite{hofi}, \cite{hoff}, for the definitions and theory of finite energy curves.

In section \ref{2pt3} we study moduli spaces ${\mathcal M}_t$ of holomorphic curves corresponding to the $W_t$. The constituent curves have area bounded above by $R-(a_2+b)$ and a monotonicity theorem as in \cite{hind} implies that ${\mathcal M}_1$ is empty. On the other hand we show that ${\mathcal M}_0$ has a single element.
To complete the proof of Theorem \ref{polylike} we prove a compactness theorem, following \cite{BEHWZ}, showing that ${\mathcal M}_0$ and ${\mathcal M}_1$ must be cobordant. This gives the required contradiction.

The proof of Theorem \ref{polydisk} is very similar to that of Theorem \ref{polylike}, although there are more closed orbits to analyze on the boundary of a polydisk itself. For simplicity, in this paper we focus on the case of polylike domains, although we outline the notational changes necessary to establish Theorem \ref{polydisk} in section \ref{polydiskcase}.

We prove Theorem \ref{extn} in section \ref{3pt1}. Although the conclusion is harder to state rigorously, the method is actually fairly general and applies not just to the nonisotopic embeddings described in this paper but to any nonisotopic domains distinguished by a Symplectic Field Theory style invariant, see \cite{EGH}. That is, suppose that a well defined $0$ dimensional moduli space of holomorphic curves can be associated to embeddings of a domain $Q$. Here the holomorphic curves map to some $(\CC P^2 \times \RR^{2(n-2)}) \setminus W$ as above, where $W$ is a smoothing of our embedded $Q$. Also suppose that isotopic domains are associated to cobordant moduli spaces. Let ${\mathcal M}_0$ and ${\mathcal M}_1$ be the moduli spaces associated to embeddings $f_0$ and $f_1$. Then if $f_0$ and $f_1$ both extend to an ellipsoid $E$, the moduli spaces ${\mathcal M}_0$ and ${\mathcal M}_1$ will automatically be cobordant.

In section \ref{3pt2} we give some examples of nonextension results, including Proposition \ref{xtn} $(iii)$.

\end{section}

\begin{section} {Finite energy curves.} \label{polyproof}

This section gives a proof by contradiction of Theorem \ref{polylike}. Some preliminary analysis is carried out in subsections \ref{2pt1} and \ref{2pt2}, then we complete the proof in subsection \ref{2pt3}.

\begin{subsection} {Approximation of $Q$.} \label{2pt1}

Here we describe our smooth approximation $W$ of $Q=Q(b,a_2, \dots ,a_n)$, together with the closed characteristics on the boundary $\partial W$. The analysis is similar to that in \cite{hind}, section $2.1$.

We start by fixing $\delta$ and $\epsilon$ with $0 < \delta << \epsilon$. Recall that our argument will be by contradiction and so we are assuming that there exists a symplectic isotopy $Q_t\subset B^4(R) \times \RR^{2(n-2)}$ with $Q_0=Q$ and $Q_1 \subset B^4(S) \times \RR^{2(n-2)}$, where $R<2a_2+b$ and $S<a_2+b$. We will need to assume that $\epsilon$ is small relative to both $a_2+b-S$ and $2a_2+b-R$. Also, by perturbing the $a_j$ if necessary, we may assume that $\epsilon$, $1/ \epsilon$ and the $1 / a_j$ are linearly independent over the rationals.

Now we choose a function $f:[0,b] \to [0,1]$ with $f(0)=0$, $f(b)=1$, $f'(x), f''(x) \ge 0$ and with the property that there exists an $x_0$ such that $f'(x)=\epsilon$ for $x<x_0 - \delta$ and $f'(x)=\frac{1}{\epsilon}$ for $x>x_0+\delta$.

Given this, we define $W$ as follows. It will be convenient to use symplectic polar coordinates on $\RR^{2n}=\CC^n$, so we set $R_j = \pi |z_j|^2$ and $\theta_j = \arg z_j \in S^1$.

$$W=\{f(R_1)+ \sum_{j=2}^n \frac{R_j}{a_j} \le 1\}.$$

The boundary $\partial W$ is foliated by the Lagrangian tori $L_c = \{R_j = c_j\}$ which degenerate precisely when some of the $R_j=0$. However, using the coordinates $\theta_j$ we can identify the nondegenerate $L_c$ with a fixed torus $T^n$ and the integer homology with $H_1(T^n, \ZZ) = \ZZ^n$.

The characteristic foliation $\ker \omega|_{\partial W}$ is generated by the (Reeb) vectorfield
$$R_W = f'(R_1) \frac{\partial}{\partial \theta_1} + \sum_{j=2}^n \frac{1}{a_j} \frac{\partial}{\partial \theta_j}.$$
In particular the Reeb vectorfield is tangent to the Lagrangian toric fibers $L_c$.

The Reeb vectorfield has two kinds of periodic orbits. The first are the elliptic orbits $\gamma^k = \{z_j=0, j \neq k\} \cap \partial W$, $k=1, \dots ,n$. We use the notation $r\gamma^k$ to denote the $r$-fold cover of $\gamma^k$.

Since the $1/a_j$ are linearly independent all other periodic orbits lie in one of the complex $2$-planes $P_k = \{z_j=0, j \neq 1,k\}$ for $2 \le k \le n$. As $\epsilon$, $\frac{1}{\epsilon}$ and $\frac{1}{a_k}$ are linearly independent orbits in these planes are either elliptic or are called hyperbolic and lie in the region where $x_0 - \delta < R_1 < x_0 + \delta$.

Suppose there exists such an $R_1$ and a rational number written in lowest terms as $\frac{m}{n}$ such that $f'(R_1) = \frac{m}{na_k}$. Then the corresponding torus fiber over $c=(R_1,0 \dots ,0,a_k(1-f(R_1)),0,\dots ,0)$ (the nonzero entries are in positions $1$ and $k$) is foliated by a $1$-parameter family of periodic Reeb orbits in the homology class $(m,0 \dots ,0,n,0,\dots ,0)$. We denote these orbits by $\gamma^k_{m,n}$. The $r$-fold cover of $\gamma^k_{m,n}$ is written as $\gamma^k_{rm,rn}$.

Now, if we fix a symplectic trivialization of $T \RR^{2n} |_{\gamma}$, the tangent bundle of $\RR^{2n}$ restricted to a closed orbit $\gamma$ of $R$ of period $T$, then the derivative of the Reeb flow (extended to act trivially normal to $\partial W$) gives a map $\psi:[0,T] \to \mathrm{Symp}(2n,\RR)$, where $\mathrm{Symp}(2n,\RR)$ is the group of $2n \times 2n$ symplectic matrices. Associated to such a path is a Conley-Zehnder index $\mu(\gamma)$ defined in this case by Robbin and Salamon in \cite{rs}. The analogue of Lemma $2.2$ in \cite{hind} is the following.

\begin{lemma} \label{cz} With respect to the standard basis of $\RR ^{2n}$ the Conley-Zehnder indices are as follows.
\begin{equation}
\begin{split}
\mu(r\gamma^k)=2r+n-1+ 2\lfloor \epsilon ra_k \rfloor + 2\sum_{j \neq k} \lfloor \frac{ra_k}{a_j} \rfloor , \mathrm{if} k \neq 1 \\
\mu(r\gamma^1)=2r+n-1+ 2\sum_j \lfloor \frac{\epsilon r}{a_j} \rfloor \\
\mu(\gamma^k_{m,n})=2(m+n)+\frac{1}{2} + (n-2) + 2\sum_{j \neq k} \lfloor \frac{na_k}{a_j} \rfloor.
\end{split}
\end{equation}
\end{lemma}

\end{subsection}

\begin{subsection} {Index and area formulas.} \label{2pt2}

We compactify the open ball $\mathring{B}^4(R)$ by identifying it with the affine part of $\CC P^2(R)$, the complex projective plane with its Fubini-Study form scaled so that lines have area $R$. We are considering a symplectic isotopy $$Q_t \subset \mathring{B}^4(R) \times \RR^{2(n-2)} \subset \CC P^2(R) \times \RR^{2(n-2)}$$ which restricts to an isotopy $W_t \subset \CC P^2(R) \times \RR^{2(n-2)}$ of $W$.

Let $X_t = \CC P^2(R) \times \RR^{2(n-2)} \setminus W_t$ equipped with the restricted symplectic form. We can choose tame almost-complex structures with cylindrical ends $J_t$ on $X_t$ as in \cite{EGH} and then study finite energy curves asymptotic to closed Reeb orbits as in \cite{three}, \cite{hofa}, \cite{hofi}, \cite{hoff}. It is convenient to define the degree $d$ of a finite energy curve to be its intersection number with $\CC P^1(\infty) \times \RR^{2(n-2)}$, where $\CC P^1(\infty)$ is the line at infinity in $\CC P^2(R)$. The basic arrangement has been described in \cite{hind}, section $2.2.1$, but here we work in $\CC P^2(R) \times \RR^{2(n-2)}$ rather than $\CC P^2$.

In this subsection we give an approximate formula for the area and the virtual index formula for finite energy curves, the analogues of Lemmas $2.3$ and $2.7$ in \cite{hind}.

Let $C$ be a genus $0$ finite energy plane with $e^k$ punctures asymptotic to multiples of $\gamma^k$, $1 \le k \le n$, the $i$th one asymptotic to $r^k_i \gamma^k$, $1 \le i \le e^k$. (Here $r^k_i$ is a natural number depending upon $i$ and $k$, hopefully this is not too confusing.) Also, let $C$ have $h^k$ punctures asymptotic to hyperbolic orbits in $P_k$, $2 \le k \le n$, with the $i$th one asymptotic to $\gamma^k_{m^k_i,n^k_i}$, $1 \le i \le h^k$.

\begin{proposition}\label{area}
Up to an error of order $\epsilon$, the symplectic area of $C$ is given by $$\mathrm{area}(C) = \int_C \omega = dR - \sum_{i=1}^{e^1} r^1_i b - \sum_{k=2}^n\sum_{i=1}^{e^k} r^k_i a_k - \sum_{k=2}^n \sum_{i=1}^{h^k} (m^k_i b + n^k_i a_k).$$
\end{proposition}

Note that the formula immediately implies that any nonconstant curves (which have positive area) must have degree $d \ge 1$.

\begin{proof}
To see this we can glue a disk in $\partial W_t$ to each asymptotic end to produce a closed cycle of degree $d$ in $\CC P^2$, which has area $dR$. The areas of these disks are roughly the negative terms in our formula (the error term comes because our hyperbolic orbits lie on $\partial W_t$ rather than the singular part of $\partial Q_t$).
\end{proof}

\begin{proposition}\label{index}
The virtual index of $C$ in the space of curves with asymptotic limits allowed to vary is given by
\begin{equation}
\begin{split}
\mathrm{index(C)} = (n-3)(2-\sum_{k=1}^n e^k - \sum_{k=2}^n h^k) + 6d \\
- \sum_{i=1}^{e^1}(2r^1_i+n-1+ 2\sum_j \lfloor \frac{\epsilon r^1_i}{a_j} \rfloor) \\
- \sum_{k=2}^n \sum_{i=1}^{e^k}(2r^k_i+n-1+ 2\lfloor \epsilon r^k_i a_k \rfloor + 2\sum_{j \neq k} \lfloor \frac{r^k_i a_k}{a_j} \rfloor ) \\
- \sum_{k=2}^n \sum_{i=1}^{h^k} (2(m^k_i+n^k_i) + (n-2) + 2\sum_{j \neq k} \lfloor \frac{n^k_ia_k}{a_j} \rfloor).
\end{split}
\end{equation}
\end{proposition}

Note here that each elliptic limit not a cover of $\gamma^1$ contributes a negative term on the third line of the index formula, the limits asymptotic to $\gamma^1$ contribute negative terms on the second line, and the hyperbolic limits each contribute a negative term on the last line.

\begin{proof}
The general index formula for genus $0$ curves with $s$ negative ends is $$\mathrm{index}(C) = (n-3)(2 - s) + 2c_1(C) - \sum_{i=1}^{s} (\mu(\gamma_i) -\frac{1}{2} \mathrm{dim} V_i).$$ For this formula, see \cite{B}.
Here $c_1(C)$ is the Chern class which we have normalized to be $3d$, where $d$ is the degree, $\mu(\gamma_i)$ is the Conley-Zehnder index of the limiting Reeb orbit $\gamma_i$ corresponding to the $i$th end, and $\mathrm{dim} V_i$ is the dimension of the family of orbits containing $\gamma_i$. In our case this dimension is $0$ for an elliptic orbit and $1$ for a hyperbolic orbit. Substituting the Conley-Zehnder indices from Lemma \ref{cz} we get the formula as required.
\end{proof}

In the remainder of this subsection we record a few algebraic consequences of the area and index formulas.

\begin{lemma} \label{con1}
Suppose that a finite energy curve $C$ has degree $1$ and $\mathrm{area}(C) \le a_2$ (up to an error of order $\epsilon$). Then $C$ either has a single hyperbolic asymptotic limit $\gamma^2_{1,1}$, or all asymptotic limits are elliptic and satisfy $$b < \sum_{i=1}^{e^1} r^1_i b + \sum_{k=2}^n \sum_{i=1}^{e^k} r^k_i a_k  <2a_2+b.$$
\end{lemma}

\begin{proof}
As nonconstant curves must have positive area, the area inequality is equivalent to
$$R-a_2 \le \sum_{i=1}^{e^1} r^1_i b + \sum_{k=2}^n \sum_{i=1}^{e^k} r^k_i a_k +  \sum_{k=2}^n \sum_{i=1}^{h^k} (m^k_i a_k + n^k_i b) \le R.$$
As $a_2+b <R < 2a_2+b$ (and $\epsilon$ is small relative to the differences) this gives
$$b < \sum_{i=1}^{e^1} r^1_i b + \sum_k \sum_{i=1}^{e^k} r^k_i a_k + \sum_{k=2}^n \sum_{i=1}^{h^k} (m^k_i a_k + n^k_i b) < 2a_2+b.$$
Since $a_k > 2a_2$ for all $k \ge 3$ we see that if there exists a hyperbolic orbit it must be of type $\gamma^2_{1,1}$ and be the only asymptotic limit. On the other hand, if all limits are elliptic then they satisfy the inequality of the lemma.
\end{proof}

\begin{lemma} \label{con3}
Suppose that a finite energy curve $C$ has degree $1$, virtual index at least $-1$, and only elliptic asymptotic limits. Then it has only a single asymptotic limit, that is, $C$ is a finite energy plane.
\end{lemma}

\begin{proof}
Let $E$ be the total number of elliptic asymptotic limits. Since all terms in the sums on the second and third lines of the index formula of Proposition \ref{index} are at least $n+1$, we have the formula
$$-1 \le \mathrm{index(C)} \le (n-3)(2-E) + 6 - (n+1)E = 2(n- (n-1)E).$$
Hence $(n-1)E \le n$ and so as $n \ge 3$ we have $E \le 1$ as required.
\end{proof}

Putting the previous two lemmas together we have the following, which describes the curves we will be interested in.

\begin{lemma} \label{con2}
Suppose that a finite energy curve $C$ has degree $1$ and $\mathrm{area}(C) \le a_2$ and $\mathrm{index}(C) \ge -1$. Then $C$ is a finite energy plane asymptotic to either $\gamma^2_{1,1}$, $2\gamma^1$ or $2\gamma^2$.
\end{lemma}

\begin{proof}
By Lemmas \ref{con1} and \ref{con3}, if the curve $C$ is not asymptotic to $\gamma^2_{1,1}$ then it is a finite energy plane asymptotic to a cover of one of the $\gamma^k$, say asymptotic to $r\gamma^k$.

Suppose first that $k=1$. Then by Lemma \ref{con1} we have $b<rb <2a_2+b$ and Proposition \ref{index} again implies that $r \le 2$. Putting the two together we have $r=2$.

Next suppose that $k=2$. Again by Lemma \ref{con1} we have $b<ra_2 <2a_2+b$ and by Proposition \ref{index} we have $$\mathrm{index}(C) \le (n-3) + 6 - (2r+(n-1)).$$ As $\mathrm{index}(C) \ge -1$ this implies that $r \le 2$. By our original hypothesis in Theorem \ref{polylike} we have $a_2<b$, and combining the two inequalities gives $r=2$.

Finally suppose that $k \ge 3$. By hypothesis $a_k > 2a_2$ and so the term $2\lfloor \frac{r a_k}{a_2} \rfloor$ in the index formula contributes at least $4$. Hence $$\mathrm{index}(C) \le (n-3) + 6 - (2r+(n-1)+4) \le -2r$$ a contradiction as required.
\end{proof}

\end{subsection}

\begin{subsection} {Moduli spaces of finite energy planes.} \label{2pt3}

Let us fix an orbit $\eta_t$ of type $\gamma^2_{1,1}$ in each $\partial W_t$. Consider the corresponding moduli space
$${\mathcal M}_t = \{u: \CC \to X_t | \mathrm{degree}(u)=1, \overline{\partial}_{J_t} u =0, u \sim \eta_t \} \slash G$$
where $u \sim \eta$ means that $u$ is asymptotic at infinity to $\eta$, and $G$ is the reparameterization group of $\CC$.
The area formula of Proposition \ref{area} says that curves in ${\mathcal M}_t$ have area roughly $R-(a_2+b)$.

We will need to choose the almost-complex structure $J_0$ such that the line at infinity $\CC P^1(\infty) \times \RR^{2(n-2)}$ is complex and such that it is invariant with respect to the $T^{n-2}$ action rotating the $(z_3, \dots ,z_n)$ planes. This is possible since $W=W_0$ is invariant under the same action. A genericity assumption will also be made as explained in Lemma \ref{pf1}. The almost-complex structure $J_1$ can be assumed to be the standard product integrable structure on $(\CC P^2(R) \setminus B^4(S)) \times \RR^{2(n-2)}$ for some $S<a_2+b$, as $W_1 \subset \mathring{B}^4(a_2+b) \times \RR^{2(n-2)}$.

\begin{lemma} \label{dimension}
The virtual dimension of ${\mathcal M}_t$ is $0$.
\end{lemma}

\begin{proof}
Proposition \ref{index} gives virtual dimension $1$ for finite energy planes of degree $1$ asymptotic to an orbit of type $\gamma^2_{1,1}$. However a curve lies in ${\mathcal M}_t$ only if it is asymptotic to the specific orbit $\eta_t$, and this imposes a $1$-dimensional constraint.
\end{proof}

The moduli spaces when $t=0,1$ are easily described.

\begin{lemma}\label{pf1} There exists an almost-complex structure $J_0$ such that the moduli space ${\mathcal M}_0$ consists of a single, regular curve.
\end{lemma}

As this is a direct generalization of Lemma $2.8$ in \cite{hind}, utilizing the analysis in \cite{hindker} to extend the results to higher dimension, we restrict here to an outline.

{\it Outline of proof.}
As $J_0$ is invariant under rotations of the $(z_3, \dots ,z_n)$ planes, the $(z_1,z_2)$-plane $P_1 = \{z_3 = \dots =z_n=0\}$ is $J_0$-invariant. Hence $J_0$ can be restricted to $Y=X_0 \cap P_1$ to give an almost-complex manifold with a cylindrical end over $\partial Y := \partial W_0 \cap P_1$. The almost-complex manifold $Y$ is exactly the one studied in \cite{hind}, and elements of ${\mathcal M}_0$ lying in $Y$ form a moduli space $\tilde{{\mathcal M}}_0$ in their own right. In particular Lemma $2.8$ from \cite{hind} implies that $\tilde{{\mathcal M}}_0$ is nonempty, that is, there exists an element of ${\mathcal M}_0$ lying in $Y$. To complete the proof we will show first that there can be no more than one element of $\tilde{{\mathcal M}}_0$ and second that, for a generic choice of invariant $J_0$, all elements of ${\mathcal M}_0$ must lie in $Y$. Lemma $3.17$ in \cite{hindker} shows that, for invariant almost-complex structures, curves in $\tilde{{\mathcal M}}_0$ are regular in ${\mathcal M}_0$.

For the first part, we argue by contradiction and suppose that two distinct curves $u_0$ and $u_1$ represent equivalence classes in $\tilde{{\mathcal M}}_0$. Automatic regularity in dimension $4$ (see \cite{wendl}, Theorem $1$, or the discussion after Theorem $2.9$ in \cite{hind}) implies that $u_0$, say, can be included in a local $1$-parameter family of curves $u_t$, $-\epsilon < t < \epsilon$, with a single curve in the family asymptotic to each $\gamma^2_{1,1}$ orbit close to $\eta_0$. Meanwhile, as $u_0$ and $u_1$ are both asymptotic to $\eta_0$, on a suitable subset of the cylindrical end $(-\infty, S_0] \times \partial Y$ we can represent $u_1$ as a section $\xi$ of the normal bundle to the image of $u_0$. Furthermore, if $S_0$ is sufficiently negative, the section $\xi$ has no zeros and so defines a winding of $u_1$ about $u_0$. For this see \cite{hofi}. This winding is the same as the winding of an eigenvector of an asymptotic operator associated to the orbit $\eta_0$, and as we are dealing with a negative puncture the associated eigenvalue must be positive.

Now, the asymptotic operator acts on sections of the normal bundle to $\eta_0$ in $\partial Y$, which has an induced complex structure still called $J_0$. With respect to a basis of the normal bundle extending a tangent vector to the space of $\gamma^2_{1,1}$ orbits, the asymptotic operator takes the form $$-J_0 \frac{d}{dt} - T \begin{pmatrix} 0 & 0 \\ 0 & 1 \end{pmatrix},$$ where $T$ is the period of $\eta_0$. We see that the only eigenvectors with winding number $0$ have eigenvalues $0$ or $-T$ and so can conclude that in this basis $u_1$ must wind around $u_0$. Hence $u_1$ must intersect the $u_t$, which have winding $0$ because they are asymptotic to different orbits. However, by gluing planes inside $W_0$ the images of $u_1$ and the $u_t$ can be included in cycles of degree $1$ in $\CC P^2$, which therefore have intersection number $1$. The added planes can be assumed to have a unique (positive) intersection point at the origin and so the intersections of $u_1$ and $u_t$ contribute $0$. This contradicts positivity of intersection.

For the second part of the proof we must exclude curves in ${\mathcal M}_0$ not lying in $Y$. As $J_0$ is $T^{n-2}$ invariant, any such curves appear in $(n-2)$-dimensional families and so are certainly not regular. Hence if we are able to assume that $J_0$ is regular for ${\mathcal M}_0$ and at the same time $T^{n-2}$ invariant then no such curves exist. The proof of the existence of regular invariant almost-complex structures follows the usual regularity argument working with invariant rather than general almost-complex structures. For this to work, instead of assuming that our holomorphic curves are somewhere injective we need the stronger assumption that corresponding to each curve in ${\mathcal M}_0$ there exists an orbit of the $T^{n-2}$ action which intersects the curve in a single point, see the proof of Proposition $3.16$ in \cite{hindker}. This is automatic in our case since by positivity of intersection a degree $1$ curve must intersect $\CC P^1(\infty) \times \RR^{2(n-2)}$ exactly once transversally, and hence intersect exactly one $T^{n-2}$ orbit in $\CC P^1(\infty) \times \RR^{2(n-2)}$, in a single point.

\qed

\begin{lemma} \label{pf2} For $J_1$ chosen as above, the moduli space ${\mathcal M}_1$ is empty.
\end{lemma}

\begin{proof} This is identical to the proof of Lemma $2.11$ in \cite{hind}. Indeed, the image of any curve in ${\mathcal M}_1$ can be restricted to a curve in $(\CC P^2(R) \setminus B(S)) \times \RR^{2(n-2)}$. As the complex structure is assumed to be a product the curve projects to a holomorphic curve in $\CC P^2(R) \setminus B(S)$, and then by a monotonicity theorem, see \cite{hind}, Lemma $2.12$, we see that it has area at least $R-S$. This is a contradiction as curves in any ${\mathcal M}_t$ have area $R-(a_2+b)$.
\end{proof}

Next we consider the universal moduli space $${\mathcal M} = \{(u,t)| u: \CC \to X_t, \mathrm{degree}(u)=1, \overline{\partial}_{J_t} u =0, u \sim \eta_t, t \in [0,1] \} \slash G.$$ This has virtual dimension $1$, but to show that it is a compact $1$-dimensional manifold (the source of our contradiction) we will need some assumptions on the family of almost-complex structures $J_t$.

First of all, since the curves in ${\mathcal M}$ have degree $1$ they are not multiply covered and so we may choose a family $J_t$ so that ${\mathcal M}$ is a $1$-dimensional manifold giving a cobordism between ${\mathcal M}_0$ and ${\mathcal M}_1$. The $J_t$ can be chosen to coincide with those we already have when $t=0,1$. Indeed, $J_0$ is regular by Lemma \ref{pf1}, and since no curves in ${\mathcal M}$ lie entirely in $(\CC P^2(R) \setminus B(S)) \times \RR^{2(n-2)}$ we are free to take $J_1$ standard here and perturb elsewhere to obtain regularity if necessary.

Second, a collection of families $\{J_t\}$ of the second category is regular in the sense that any somewhere injective $J_{t_0}$-holomorphic finite energy curve, for $t_0 \in [0,1]$, has deformation index at least $-1$ (amongst $J_{t_0}$ curves). We will also assume then that our $J_t$ are regular in this sense.

Finally, the cylindrical ends of the $X_t$ are all symplectomorphic, and after identifying them by a symplectomorphism we may assume that all $J_t$ are identical outside of a compact set. This implies that they induce identical translation invariant almost-complex structures on the symplectization $S(\partial W) = \RR \times \partial W$. Holomorphic curves in $S(\partial W)$ are either translation invariant, which means they are covers of cylinders over Reeb orbits, or come in families of dimension at least $1$. Therefore, if an almost-complex structure is regular, somewhere injective finite energy curves are either trivial cylinders or have deformation index at least $1$. As above such almost-complex structures form a subset of the second category and we will assume our $J_t$ induce a structure in this class.

The final lemma is the following, which contradicts Lemmas \ref{pf1} and \ref{pf2}.

\begin{lemma} \label{compact}
The universal moduli space ${\mathcal M}$ is sequentially compact.
\end{lemma}

\begin{proof}
The general compactness theorem for finite energy curves can be found in \cite{BEHWZ}. In our situation, it implies that a sequence of finite energy curves $u_i$ representing classes in ${\mathcal M}_{t_i}$ with $t_i \to t_{\infty}$, after taking a subsequence, converge in the sense of \cite{BEHWZ} to a holomorphic building in $X_{t_{\infty}}$. For components in $X_{t_{\infty}}$ to be nonconstant they must have positive degree (see the comment after Proposition \ref{area}), and so since degree is preserved in the limit and the $u_i$ have degree $1$ our limit must consist of a single curve $u$ in $X_{t_{\infty}}$ of degree $1$. Therefore the curve is also somewhere injective. By regularity of the family of $J_t$ we have $\mathrm{index}(u) \ge -1$, and as the $u_i$ have area roughly $R-(a_2+b)$ the area of $u$ is bounded above by $R-(a_2+b)$. As $a_2<b$ by assumption, this excludes planes asymptotic to $2\gamma^2$ and hence by Lemma \ref{con2}, the curve $u$ is a finite energy plane asymptotic to either an orbit $\gamma^2_{1,1}$ or to $2\gamma^1$. In the first case, as the limit preserves area, it must be asymptotic to $\eta_{t_{\infty}}$ itself (as otherwise we would see symplectization components of positive area). Hence $(u,t_{\infty})$ represents a class in ${\mathcal M}$ and we have compactness as required.

It remains to exclude limiting planes asymptotic to $2\gamma^1$, which have index $0$ by Proposition \ref{index}. If a curve with such a limit exists then we have $R-2b>0$ and so $2b<R<2a_2+b$ and $b<2a_2<a_j$ for $j \ge 3$.

We look at components of the limit mapping to the symplectization layers $S(\partial W)$. There is a single curve in the highest level with positive end asymptotic to $2\gamma^1$. If this curve is a cylinder then the negative end is an asymptotic orbit with action between $a_2 + b$ and $2b$. Given the inequalities above, the only possibilities are negative ends on orbits of type $\gamma^2_{1,1}$ or $3 \gamma^2$ or $\gamma^j$ for some $j \ge 3$. In all three cases the greatest common divisor of the covering degrees of the positive and negative ends is $1$ and so the cylinder is somewhere injective. A variation of Proposition \ref{index} (or simply using the fact that the total index is preserved in a limit) shows that cylinders asymptotic to $3 \gamma^2$ or $\gamma^j$ for $j \ge 3$ have deformation index at most $-2$ and so we do not expect such cylinders to exist for regular almost-complex structures. Cylinders asymptotic to $\gamma^2_{1,1}$ have deformation index $1$, but by translation invariance we do not expect such a cylinder to have negative end on $\eta_{t_{\infty}}$. By area reasons, such a cylinder cannot be connected to any lower level curves, and so this possibility can also be excluded.

Finally, suppose the highest level curve in $S(\partial W)$ has several ends. As we take a limit of curves of genus $0$ exactly one of these ends is connected in our limiting building to $\eta_{t_{\infty}}$ and it has action at least $a_2 + b$. The remaining ends have action less than $2b - (a_2+b) = b - a_2 < a_2$ by the inequality above. Since no such periodic orbits exist we have a contradiction.

\end{proof}

\end{subsection}

\end{section}

\begin{section} {Isotopies of polydisks.} \label{polydiskcase}

Here we outline how the proof of Theorem \ref{polylike} can be adapted to prove Theorem \ref{polydisk}.

We argue by contradiction and suppose that there exists an isotopy $f_t:P(a_1, \dots ,a_n) \to B^4(R) \times \RR^{2(n-2)}$ with $f_0=f$ as in Theorem \ref{polylike} and $f_1(P(a_1, \dots ,a_n)) \subset \mathring{B}^4(a_1+a_3) \times \RR^{2(n-2)}$. Let $a=a'_1=a_1$ and $b=a'_3=a_3>2a$ and $a'_2=a'_4= \dots a'_n=a-\epsilon$, where $2\epsilon < 2a+b-R$. Then our isotopy restricts a polydisk $P=P(a'_1, \dots a'_n)$. We will show that no such isotopy of $P$ can exist.

The proof begins by smoothing $P$ to a domain $W$ in the same way as we perturbed a polylike domain in section \ref{2pt1}. However $\partial W$ will now contain more families of closed Reeb orbits. These can be described as follows. Let $m_1, \dots ,m_k$ be positive integers and $I$ a subset of $k$ distinct integers from $\{1, \dots ,n\}$. Then $\gamma^I_{m_1, \dots ,m_k}$ denotes a $(k-1)$-dimensional family of Reeb orbits approximating curves $[0,2\pi] \to \partial P$ given by  $$t \mapsto (\delta^I_1 a'_1 e^{i(\phi_1+t)}, \dots ,\delta^I_k a'_k e^{i(\phi_k+t)}).$$ Here $\delta^I_i=1$ if $i \in I$ and $0$ otherwise.

We may assume that our family of embeddings $f_t$ extend to $W$ and set $W_t=f_t(W)$. As in section \ref{2pt3}
we fix an orbit $\eta_t$ of type $\gamma^{1,3}_{1,1}$ in each $\partial W_t$. Then compactifying $\mathring{B}^4(R)$ to $\CC P^2(R)$ as in section \ref{2pt2}, we define $X_t = \CC P^2(R) \times \RR^{2(n-2)} \setminus W_t$, choose compatible almost-complex structures $J_t$, and study the corresponding moduli spaces
$${\mathcal M}_t = \{u: \CC \to X_t | \mathrm{degree}(u)=1, \overline{\partial}_{J_t} u =0, u \sim \eta_t \} \slash G.$$
Lemma \ref{dimension}, saying that ${\mathcal M}_t$ has dimension $0$, remains true in this setting, and Lemma \ref{pf1}, saying that ${\mathcal M}_0$ has a single element, also holds, with the same proof. As curves in ${\mathcal M}_1$ have area $R-(a'_1+a'_3)=R-(a_1+a_3)$, the monotonicity Lemma \ref{pf2} also holds here to say that ${\mathcal M}_1$ is empty. Hence our proof again boils down to showing that the universal moduli space $${\mathcal M} = \{(u,t)| u: \CC \to X_t, \mathrm{degree}(u)=1, \overline{\partial}_{J_t} u =0, u \sim \eta_t, t \in [0,1] \} \slash G.$$ is compact.

Limiting buildings whose components in $X_t$ have multiple ends can be excluded using area inequalities as in section \ref{polyproof}. Therefore, following Lemma \ref{compact}, we need to study $J_t$-holomorphic degree $1$ components $u: \CC \to X_t$ of a holomorphic building in the boundary of ${\mathcal M}$. Suppose $u$ is asymptotic to an orbit of type $\gamma^I_{m_1, \dots ,m_k}$. As area is preserved in limits and $a'_1+a'_3 = a+b$ we have $$a+b \le \sum_i m_i a'_i = am_1 + bm_3 + (a-\epsilon) \sum_{i \neq 1,3} m_i \le R.$$
As $u$ has degree $1$ it is somewhere injective and genericity assumptions here imply that $\sum_i m_i \le 3$. Then as $a+b>3a$ and $2(a-\epsilon) +b > R$ this implies that in fact $m_1=m_3=1$ and all other $m_i=0$. As in Lemma \ref{compact} this implies that $u \in  {\mathcal M}$ as required.

\end{section}

\begin{section} {Extension to ellipsoids.} \label{ellipsoids}

\begin{subsection} {The proof of Theorem \ref{extn}.}\label{3pt1}

In this section we prove Theorem \ref{extn}. For clarity we will restrict to part $(i)$, although as mentioned in the introduction the method is actually quite general.

Suppose then that $f_0,f_1:E \to B^4(R) \times \RR^{2(n-2)}$ are symplectic embeddings which restrict to embeddings of a polylike domain $Q$. Here $E=E(c_1, \dots ,c_n)$ is an ellipsoid, say with $c_1 \le \dots \le c_n$, and $Q=Q(b,a_2, \dots ,a_n)$ is a polylike domain for which $b,a_2, \dots ,a_n,R$ satisfy the hypotheses of Theorem \ref{extn} $(i)$.

We will use the following.

\begin{lemma} \label{ekho} If there exists a symplectic embedding $E \to B^4(R) \times \RR^{2(n-2)}$ then $\min (2c_1,c_2) \le R$.
\end{lemma}

\begin{proof} This is an application of the Ekeland-Hofer capacities, see \cite{ekehof}. Indeed, the second Ekeland-Hofer capacity of $E$ is $\min (2c_1,c_2)$, while the second Ekeland-Hofer capacity of $B^4(R) \times \RR^{2(n-2)}$ is $R$. As these capacities are monotonic under embeddings the proof follows.
\end{proof}

Corresponding to the embeddings $f_0|_Q$ and $f_1|_Q$, together with various choices including a smoothing of the image, a choice of asymptotic Reeb orbit $\eta$, and a compatible almost-complex structure, we can define moduli spaces ${\mathcal M}_0$ and ${\mathcal M}_1$ as in section \ref{2pt3}. Our goal is to show that ${\mathcal M}_0$ and ${\mathcal M}_1$ are cobordant.

We first observe that without the constraint of the image lying in $B^4(R) \times \RR^{2(n-2)}$, it is easy to construct a symplectic isotopy between $f_0$ and $f_1$. That is, there exists an $S \ge R$ and a family of symplectic embeddings $f_t:E \to B^4(S) \times \RR^{2(n-2)}$ interpolating between $f_0$ and $f_1$. These embeddings restrict to give an isotopy of $Q$.

As in section \ref{2pt2}, let $W_t$ be a smoothing of $f_t(Q)$ and $\{J_t\}$ be a family of compatible almost-complex structures on $X_t = \CC P^2(S) \times \RR^{2(n-2)} \setminus W_t$. Recall that $\CC P^2(S)$ denotes the compactification of the ball $\mathring{B}^4(S)$. Given this, following section \ref{2pt3}, we have a universal moduli space $${\mathcal M} = \{(u,t)| u: \CC \to X_t, \mathrm{degree}(u)=1, \overline{\partial}_{J_t} u =0, u \sim \eta_t, t \in [0,1] \} \slash G$$ whose boundary is the disjoint union of ${\mathcal M}_0$ and ${\mathcal M}_1$. To complete the proof we will show that for a suitable choice of $J_t$ the moduli space ${\mathcal M}$ is compact, and hence a cobordism between ${\mathcal M}_0$ and ${\mathcal M}_1$.

We choose a family $J^N_t$ exactly as in section \ref{2pt2} but with the additional condition that if $f_t(E) \not \subset \mathring{B}^4(R) \times \RR^{2(n-2)} \subset B^4(S) \times \RR^{2(n-2)}$ then $J^N_t$ is stretched to length $N$ along $f_t(\partial E)$. On the other hand, if  $f_t(E) \subset \mathring{B}^4(R) \times \RR^{2(n-2)}$ then we require $J^N_t$ to be the standard product integrable structure on $(\CC P^2(S) \setminus B^4(R)) \times \RR^{2(n-2)}$. As $f_0(E)$ and $f_1(E)$ lie in $B^4(R) \times \RR^{2(n-2)}$ this condition leaves us free to choose $J^N_0$ and $J^N_1$ as in section \ref{2pt3}. In particular, arguing by contradiction in the case of Theorem \ref{extn} $(i)$, we can choose $J_0$ such that ${\mathcal M}_0$ has a single element and $J_1$ such that ${\mathcal M}_1$ is empty.

We claim that for such $J^N_t$, with $N$ chosen sufficiently large, the moduli space ${\mathcal M}$ is compact. The line of argument follows that of Lemma \ref{compact}. Suppose that a sequence of curves $u_n \in {\mathcal M}$ converges to a holomorphic building, which will be $J^N_t$ holomorphic for some $t \in [0,1]$. As the $u_n$ all have degree $1$, our limiting building will have a single component $u$ in $X_t$ which is also of degree $1$ and so somewhere injective. It is required to show that $u$ represents an element of ${\mathcal M}$.

Now, the $u_n$ have area roughly $S-(a_2+b)$ and so the area of $u$ is bounded above by $S-(a_2+b)$ and the action of its negative ends is bounded above by $S$ and below by $a_2+b$. We can apply Lemma \ref{compact} once we show that in fact this action is bounded above by $R$, or equivalently that the area of $u$ is bounded below by $S-R$. Indeed, then the only possible limits are as described in Lemma \ref{compact}, see also Lemma \ref{con2}, and the proof follows.

We argue by contradiction and suppose that for all large $N$ we can find a $J^N_t$ holomorphic limiting component $u$ in $X_t$ with area less than $S-R$.

If the almost-complex structure $J^N_t$ is standard on $(\CC P^2(S) \setminus B^4(R)) \times \RR^{2(n-2)}$ then by the monotonicity theorem as in \cite{hind}, Lemma $2.12$ (see Lemma \ref{pf2} above), we can conclude that $u$ must have area at least $S-R$, a contradiction.

Hence, for all large $N$ we can find a degree $1$ curve $u^N$ in some $X_t$, with area less than $S-R$ and with respect to an almost-complex structure $J^N_t$ stretched to length $N$ along $f_t(\partial E)$. We will take a limit of these curves $u^N$ as $N \to \infty$ and see that the area of the limiting component $v$ in $X_t \setminus f_t(E)$ is at least $S - R$. This implies that the $u^N$, at least for large $N$, also have area at least $S - R$, giving our contradiction.

To investigate this limit we need to review the structure of the Reeb orbits on $\partial E$. Without loss of generality we may assume that the $c_i$ are rationally indepenent. Then there are $n$ geometrically disctinct closed orbits on $\partial E$, namely $\delta^k = \partial E \cap \{z_i=0, i \neq k\}$. The $r$ fold cover of $\delta^k$ has Conley-Zehnder index $$\mu(r \delta^k) = 2r + (n-1) + 2\sum_{j \neq k} \lfloor \frac{rc_k}{c_j} \rfloor.$$

Suppose that $v$ has a total of $s$ negative ends, with $s_k$ being covers of $\delta^k$ and the $i$th of these covering $\delta^k$ with multiplicity $r^k_i$. Then $v$ has virtual index given by the formula
\begin{equation*}
\begin{split}
\mathrm{index(v)} = (n-3)(2-s) + 6
- \sum_{k=1}^{n}\sum_{i=1}^{s_k}(2r^k_i+n-1+ 2\sum_{j \neq k} \lfloor \frac{rc_k}{c_j} \rfloor) \\
=(n-3)(2-2s) + 6 - 2s - 2\sum_{k=1}^{n}\sum_{i=1}^{s_k}(r^k_i + \sum_{j \neq k} \lfloor \frac{rc_k}{c_j} \rfloor).
\end{split}
\end{equation*}

As we work with a $1$-parameter family of almost-complex structures and $u$ is somewhere injective (as it has degree $1$), we may assume that this index is at least $-1$. This eliminates most of the possibilities for the ends of $v$. Indeed, we see that $v$ must have a single negative end, which is asymptotic to either $\delta^1$, the double cover $2\delta^1$, or $\delta^2$. Hence $v$ has area $S-c_1$, $S - 2c_1$ or $S-c_2$. Moreover, if the end is asymptotic to $2\delta^1$ then we must have $2c_1<c_2$, and by Lemma \ref{ekho} this means that $2c_1<R$. Similarly, if the end is asymptotic to $\delta^2$ then we must have $c_2<2c_1$, and hence Lemma \ref{ekho} implies $c_2<R$. In all cases then, the area of $v$ is bounded below by $S-R$ and our proof is compete.

\end{subsection}

\begin{subsection} {Some symplectic embeddings.}\label{3pt2}

In this section we give some examples of the general nonextension result described in Theorem \ref{extn}. Everything here is a fairly direct consequence of Theorem \ref{extn}, but we think there is some insight into the nature of symplectic embeddings.

We will work with a specific $6$-dimensional polylike domain $Q=Q(b,1,2)$ with $b>1$. Then by Theorem \ref{polylike} we have the following.

\begin{theorem}\label{thm31} Let $b+1<R<b+2$. Then the two embeddings $f_0,f_1:Q \to B^4(R) \times \RR^2$ given by $f_0(z_1,z_2,z_3)=(z_1,z_2,z_3)$ and $f_1(z_1,z_2,z_3)=(z_2,z_3,z_1)$ are not isotopic.
\end{theorem}

Next observe the following.

\begin{lemma}\label{outer} If $A>1$, $B=\frac{bA}{A-1}$ and $C=2A$, then $Q \subset E(B,A,C)$.
\end{lemma}

\begin{proof} If $(z_1,z_2,z_3) \in Q$ then we have
\begin{equation*}
\begin{split}
\frac{\pi|z_1|^2}{B} + \frac{\pi|z_2|^2}{A} + \frac{\pi|z_3|^2}{C} \le \frac{b}{B} + \frac{1}{A}(\pi|z_2|^2 + \frac{\pi|z_3|^2}{2}) \\
\le \frac{b}{B} + \frac{1}{A} \le 1.
\end{split}
\end{equation*}

\end{proof}

Now, with $A,B,C$ as in Lemma \ref{outer}, if $A,B<b+2$, or equivalently $\frac{b+2}{2} < A < b+2$, then by inclusion we have $E(B,A,C) \subset \mathring{B}^4(2+b) \times \RR^2$. Hence the map $f_0$ from Theorem \ref{thm31} extends (as the inclusion) to an embedding of $E$. Therefore we can apply Theorem \ref{extn} to say the following.

\begin{proposition} \label{prop32} Let $\frac{b+2}{2} < A < b+2$. The map $f_1:Q \to \mathring{B}^4(b+2) \times \RR^2$, $(z_1,z_2,z_3) \mapsto (z_2,z_3,z_1)$ does not extend to an embedding of $E$.
\end{proposition}

Proposition \ref{prop32} is certainly an extension theorem (rather than an embedding obstruction) as we know that an embedding of $E$ exists, namely the inclusion $f_0$. However, there is no embedding of $E$ of the form $(z_1,z_2,z_3) \mapsto (g_1(z_2,z_3),z_1)$, since there exists a map $g_1:E(A,2A) \to \mathring{B}^4(b+2)$ if and only if $A < \frac{b+2}{2}$. This is a consequence for example of the Ekeland-Hofer capacities, \cite{ekehof}. In other words, Proposition \ref{prop32} says nothing about extensions of embeddings in dimension $4$.

Nevertheless we can obtain a new extension result for $4$-dimensional embeddings, saying that the obstructions to the extension of an embedding to an ellipsoid can be partially localized. The following is a generalization of Proposition \ref{xtn} $(iii)$.

\begin{proposition}\label{prop33} Let $\frac{b+2}{2} < A < b+1$. The inclusion map $E(A,2A) \cap \{\pi|z_2|^2 =2 \} \to \mathring{B}^4(b+2)$ does not extend to a symplectic embedding $E(A,2A) \cap \{\pi|z_2|^2 \ge 2 \} \to \mathring{B}^4(b+2)$.
\end{proposition}

Before proving this, we remark that this is an extension result in the sense that embeddings of $E(A,2A) \cap \{\pi|z_2|^2 \ge 2 \}$ exist, at least for some $A$, as shown in the following.

\begin{lemma}\label{lastlem} If $A < \frac{b+3}{2}$ then there exists an embedding $E(A,2A) \cap \{\pi|z_2|^2 \ge 2 \} \to \mathring{B}^4(b+2)$.
\end{lemma}

\begin{proof} Choose an $\tilde{A}$ with $A - \frac{1}{2} < \tilde{A} < \frac{b+2}{2}$. This is possible by our hypothesis on $A$. Then we have $E(A,2A) \cap \{\pi|z_2|^2 \ge 2 \} \subset E(\tilde{A}, 4\tilde{A})$. Indeed, if $(z_1,z_2) \in E(A,2A) \cap \{\pi|z_2|^2 \ge 2 \}$ then
\begin{equation*}
\begin{split}
\frac{\pi|z_1|^2}{\tilde{A}} + \frac{\pi|z_2|^2}{4\tilde{A}} =  \frac{A}{\tilde{A}}(\frac{\pi|z_1|^2}{A} + \frac{\pi|z_2|^2}{4A}) \\
= \frac{A}{\tilde{A}}(\frac{\pi|z_1|^2}{A} + \frac{\pi|z_2|^2}{2A} - \frac{\pi|z_2|^2}{4A}) \\
\le \frac{A}{\tilde{A}}(1-\frac{2}{4A}) = \frac{A-\frac{1}{2}}{\tilde{A}} <1.
\end{split}
\end{equation*}

Finally, there exists an embedding $E(\tilde{A}, 4\tilde{A}) \to B^4(2\tilde{A}) \subset B^4(b+2)$. This follows from the classification of ellipsoid embeddings into balls contained in \cite{mcdsch}, although this particular embedding was also known at least to Opshtein, \cite{opshtein} Lemma $2.1$.
\end{proof}

\begin{proof}( of Proposition \ref{prop33}.) Let $E=E(B,A,C)$ as above, that is $\frac{b+2}{2} < A < b+1$, $B=\frac{bA}{A-1}$ and $C=2A$. Again we let $Q=Q(b,1,2)$ be the polylike domain inside $E$, and look at the map $f_1$ as above.
We will argue by contradiction and show that if an extension $g$ exists as in Proposition \ref{prop33} then an extension of $f_1|_Q$ to a map $E \to \mathring{B}^4(b+2) \times \RR^2$ must also exist, contradicting Proposition \ref{prop32}.

We first note that $f_1(E \cap \{\pi|z_3|^2 \le 2 \}) \subset \mathring{B}^4(b+2) \times \RR^2$. This is because $f_1(z_1,z_2,z_3)=(z_2,z_3,z_1)$ and if $(z_1,z_2,z_3) \in E \cap \{\pi|z_3|^2 \le 2 \}$ then $$\pi|z_2|^2 + \pi|z_3|^2 = A(\frac{\pi|z_2|^2}{A} + \frac{\pi|z_3|^2}{2A}) +  \frac{\pi|z_3|^2}{2} \le A+1 < b+2.$$

Next we consider the action of a Hamiltonian diffeomorphism $h$ on the domain $f_1(E)=E(A,2A,B)$. The diffeomorphism will restrict to the identity on $f_1(Q)=E(1,2) \times B^2(b) \subset E(A,2A,B)$.

Let $\chi:[0,\infty) \to [0,\infty)$ be a cut-off function with $\chi(x)=0$ if $x<2$ and $\chi(x)=1$ if $x >2+\epsilon$, for a small $\epsilon$. Then for $K$ large we define the Hamiltonian function $$H(z_1,z_2,z_3)=Kx_3\chi(\pi|z_2|^2).$$ Here we are denoting by $x_3$ and $y_3$ the real and imaginary parts of $z_3$. The resulting time $1$ flow $h$ of the Hamiltonian vector field corresponding to $H$ leaves the region $\{\pi|z_2|^2 \le 2\}$, and in particular $f_1(Q)$, pointwise fixed. It also preserves $|z_1|$ and $|z_2|$ and the coordinate $x_3$, but the $y_3$ coordinate of points with $\pi|z_2|^2 >2$ is increased under $h$. Regarding level sets of $h(f_1(E))$ we can say the following.

\begin{enumerate}[I.]
\item if $d \ge B$ then $h(f_1(E)) \cap \{z_3=c+id\} \subset E(A,2A) \cap \{\pi|z_2|^2 >2\}$;
\item if $d \le K-B$ then $h(f_1(E)) \cap \{z_3=c+id\} \subset E(A,2A) \cap \{\pi|z_2|^2 <2+\epsilon\}$.
\end{enumerate}

For $\epsilon$ sufficiently small this second region lies in $\mathring{B}^4(b+2) \times \RR^2$ by the computation above.

We may assume that our extension $g$ is actually the inclusion on a narrow domain $E(A,2A) \cap \{2 \le \pi|z_2|^2 \le 2+\epsilon \}$. Then consider the map
$g \times \mathrm{id}:(z_1,z_2,z_3) \mapsto (g(z_1,z_2),z_3)$ restricted to the portion of $\CC^3$ with $y_3 \ge K-B$. By our hypothesis and point I. above, this maps $h(f_1(E)) \cap \{y_3 \ge K-B\} \to \mathring{B}^4(b+2) \times \CC$. But by point II. our map extends as the identity to the remainder of $h(f_1(E))$, and in particular is the identity on $h(f_1(Q))=f_1(Q)$. Thus we have an extension of the embedding $f_1$ of $Q$, and this contradicts Proposition \ref{prop32} as required.
\end{proof}

\end{subsection}

\end{section}

\end{document}